\documentclass[a4paper,10pt]{article}
\usepackage[utf8]{inputenc}
\usepackage{amsmath}
\usepackage{amssymb}
\usepackage{amsthm}
\usepackage{dsfont}
\usepackage{hyperref}
\usepackage{graphicx}
\usepackage{enumerate}
\usepackage{booktabs}
\usepackage{arydshln}
\usepackage{mathtools}

\newcommand{\R}{\mathbb{R}}
\newcommand{\E}{\mathbb{E}}
\newcommand{\Per}{\mathcal{P}}
\newcommand{\Opp}{\mathcal{O}}

\newtheorem{theorem}{Theorem}[section]
\newtheorem{proposition}[theorem]{Proposition}
\newtheorem{lemma}[theorem]{Lemma}
\newtheorem{definition}[theorem]{Definition}

\newtheorem{remark}[theorem]{Remark}

\title{On Majorization in Dependence Modeling}
\author{Michael Preischl}

\begin{document}

\maketitle
{\abstract{
We apply concepts of majorization theory to derive new insights in the field of extremal dependence structures. In particular, we consider the Rearrangement Algorithm by Puccetti and Rüschendorf, where majorization arguments yield a statement that unifies and extends the existing theory in two ways. The first extension considers convex functions of non-linear risk aggregation and the second allows for non-symmetric cost functions. The article is concluded by computing an example.}}
\section{Introduction and existing theory}
When talking about dependence uncertainty, a major interest usually lies on finding a dependence structure that maximizes or minimizes a certain expectation. That is, given distributions $F^{(1)},\dots,F^{(d)}$ and a function $f:\R^d\to\R$, what are
\begin{align*}
 &\inf\left\{\E\left[f(X^{(1)},\dots,X^{(d)})\right]:X^{(i)}\sim F^{(i)}, 1\leq i\leq d\right\}\quad\text{and}\\
 &\sup\left\{\E\left[f(X^{(1)},\dots,X^{(d)})\right]:X^{(i)}\sim F^{(i)}, 1\leq i\leq d\right\}?
\end{align*}
For a start, assume that $X^{(1)},\dots,X^{(d)}$ are discrete and take $n$ not necessarily distinct values with an equal probability of $\frac1n$. In this case $X^{(i)}$ is said to have an \textit{$n$-discrete} distribution for $i=1,\dots,d$. Denote by $x^{(i)}=(x^{(i)}_1,\dots,x^{(i)}_n)^T$ a vector containing the possible values for each $X^{(i)}$. Thus the $n\times d$ matrix $X=(x^{(1)},\dots,x^{(d)})$ can be interpreted as the joint distribution of a random vector $X=(X^{(1)},\dots,X^{(d)})$, where $X^{(i)}\sim F^{(i)}$ by giving each of its rows equal probability $\frac1n$. Obviously, for this distribution, we have
\[
 \E[f(X)]=\frac1n\sum_{k=1}^nf(x^{(1)}_k,\dots,x^{(d)}_k).
\]
If we now rearrange the entries in one or multiple columns of $X$, we get a new distribution $\tilde X$ which has the same marginals as $X$. Write
\[
\mathcal{P}(X)=\{(\tilde{x}^{(1)},\dots,\tilde{x}^{(d)}): \tilde{x}^{(i)}=\pi_i x^{(i)}\text{, }\pi_i\text{ is a permutation on }\{1,\dots,n\}\}
\]
for the set of matrices that can be obtained from $X$ by permuting elements within columns. In the following we will try to find
\[
m_f(X):=\min_{\tilde{X}\in\mathcal{P}}\sum_{k=1}^nf(\tilde{x}^{(1)}_k,\dots,\tilde{x}^{(d)}_k).
\]
Note that the factor $\frac1n$ was omitted since it does not affect optimality. This is the setting, Puccetti and Rüschendorf considered in \cite{puccetti2015expectation}. They give conditions under which the minimum is attained at a matrix which is in some sense oppositely ordered. More precisely, assume $h:\R^d\to\R$ can be decomposed into two functions $h^2:\R^2\to\R$ and $h^{d-1}:\R^{d-1}\to\R$ such that for all $x\in\R^d$
\begin{equation}\label{decompose_property}
 h(x_1,\dots,x_d)=h^2(x_i,h^{d-1}(x_1,\dots,x_{i-1},x_{i+1},\dots,x_d))\qquad\forall i=1,\dots,d.
\end{equation}
We write $X_{-i}$ for the matrix $(x^{(1)},\dots,x^{(i-1)},x^{(i+1)},\dots,x^{(d)})$, i.e. $X$ with the $i$th column deleted and
\[
 h^{d-1}(X_{-i}):=\begin{pmatrix}
       h^{d-1}\big(x^{(1)}_1,\dots,x^{(i-1)}_1,x^{(i+1)}_1,\dots,x^{(d)}_1\big)\\
       h^{d-1}\big(x^{(1)}_2,\dots,x^{(i-1)}_2,x^{(i+1)}_2,\dots,x^{(d)}_2\big)\\
       \vdots\\
       h^{d-1}\big(x^{(1)}_n,\dots,x^{(i-1)}_n,x^{(i+1)}_n,\dots,x^{(d)}_n\big)
      \end{pmatrix}
\]
for the vector that is $h^{d-1}$ applied to every row of $X_{-i}$.\\
Then define
\[
\Opp_h(X):=\{\tilde{X}\in\Per(X):\tilde{x}^{(i)}\perp h^{d-1}(\tilde{X}_{-i})\text{, }i=1,\dots,d\}.
\]
Here $x\perp y$ denotes that the vectors $x$ and $y$ are oppositely ordered, meaning: for $x,y\in\R^n: x\perp y \iff (x_i-x_j)(y_i-y_j)\leq 0$ for all $i,j\in\{1,\dots,n\}$.
A particular choice for $h$ which satisfies \eqref{decompose_property} would be the sum operator, i.e. $h(x_1,\dots,x_d)=x_1+\dots+x_d$. 
Now Puccetti and Rüschendorf state two cases in which the minimum is attained at an element of $\Opp_h(X)$ (see Propositions $2.4$ and $2.6$ from \cite{puccetti2015expectation}):
\begin{theorem}\label{Pucc2.4}
 If $f(x^{(1)},\dots,x^{(d)})=g(h(x^{(1)},\dots,x^{(d)}))$ where $g:\R\to\R$ is convex and $h(x_1,\dots,x_d)=x_1+\dots+x_d$ is the sum operator, then
 \[
  m_f(X)=\min_{\tilde{X}\in\Opp_h(X)}\sum_{k=1}^nf(\tilde{x}^{(1)}_k,\dots,\tilde{x}^{(d)}_k).
 \]
\end{theorem}
\begin{theorem}\label{Pucc2.6}
 If $h:\R^d\to\R$ is supermodular, componentwise strictly monotonic and satisfies \eqref{decompose_property} with $h^2$ and $h^{d-1}$ also being supermodular, then
 \[
  m_h(X)=\min_{\tilde{X}\in\Opp_h(X)}\sum_{k=1}^nh(\tilde{x}^{(1)}_k,\dots,\tilde{x}^{(d)}_k).
 \]
\end{theorem}$ $\\
Remember that a function $h:\R^d\to \R$ is called \textit{supermodular} if\\
$h(x)+h(y)\leq h(x\wedge y)+h(x\vee y)$, where $\wedge$ and $\vee$ denote the componentwise minimum respectively maximum. Some authors prefer the term \textit{$L$-superadditive} instead of supermodular.\\
The concept (and the implementation) of rearranging the matrix $X$ until obtaining an element in $\mathcal{O}_h(X)$ is the rearrangement algorithm (RA).
$ $\\
\section{Main results}
As is noted in \cite{puccetti2015expectation}, Theorem \ref{Pucc2.4} and Theorem \ref{Pucc2.6} consider two distinct cases, though the difference might seem subtle. We will clarify this later, but for the moment, we want to mention that a minor extra assumption allows us to unify these cases.
\begin{theorem}\label{weak_statement}
 If $f(x^{(1)},\dots,x^{(d)})=g(h(x^{(1)},\dots,x^{(d)}))$ where $g:\R\to\R$ is increasing and convex and $h:\R^d\to\R$ is supermodular, componentwise strictly monotonic and satisfies \eqref{decompose_property} with $h^2$ also being supermodular, then 
 \[
  m_f(X)=\min_{\tilde{X}\in\Opp_h(X)}\sum_{k=1}^nf(\tilde{x}^{(1)}_k,\dots,\tilde{x}^{(d)}_k).
 \]
\end{theorem}$ $\\
Note that Theorem \ref{Pucc2.6} is trivially included in \ref{weak_statement}, whereas \ref{Pucc2.4} is not due to the extra monotonicity assumption on $g$. A proof of Theorem \ref{weak_statement} will be given later along with the remark that it is actually possible to derive a stronger conclusion from the assumptions of \ref{weak_statement} than the one given in the Theorem. To see this, we need some notions from the theory of majorization, which are taken from the extremely rich summary by Marshall et al. \cite{marshall2011inequalities}, where also the original authors of the statements cited in this section can be found.\\ 
$ $\\
The concept of majorization allows to find extreme values for certain functions $f:\R^n\to\R$ based on a (partial) ordering of the vectors $x\in\R^n$. Let $(x_{[1]},\dots,x_{[n]})$ denote the decreasing rearrangement of $x$, i.e. $x_{[1]}$ is the largest component of $x$ and $x_{[n]}$ is the smallest. Analogously, we write $(x_{(1)},\dots,x_{(n)})$ for the increasing rearrangement of $x$, i.e. $x_{[1]}=x_{(n)}$ and $x_{[n]}=x_{(1)}$. For reasons of convenience, we will also use the notation $x_\downarrow:=(x_{[1]},\dots,x_{[n]})$ and $x_\uparrow:=(x_{(1)},\dots,x_{(n)})$ to denote the decreasing and increasing rearrangements. Obviously it holds $x_\downarrow\perp x_\uparrow$.
\begin{definition}
 A vector $x=(x_1,\dots,x_n)$ is said to be majorized by $y=(y_1,\dots,y_n)$ (write $x\prec y$) if
 \begin{align*}
  \sum_{i=1}^kx_{[i]}&\leq\sum_{i=1}^ky_{[i]}\qquad k=1,\dots,n-1\\
  \sum_{i=1}^nx_{[i]}&=\sum_{i=1}^ny_{[i]}.
 \end{align*}
 Furthermore, $x$ is said to be weakly submajorized by $y$ (write $x\prec_w y$) if
 \[
  \sum_{i=1}^kx_{[i]}\leq\sum_{i=1}^ky_{[i]}\qquad k=1,\dots,n.
 \]
 Finally, $x$ is said to be weakly supermajorized by $y$ (write $x\prec^w y$) if
 \[
  \sum_{i=1}^kx_{(i)}\geq\sum_{i=1}^ky_{(i)}\qquad k=1,\dots,n.
 \]
\end{definition}$ $\\
Obviously, we have $x\prec y$ if and only if $x \prec_w y$ and $x\prec^w y$.\\
$ $\\
When we observe $x\prec y$, we are naturally interested in the effects of this majorization. This leads to \textit{Schur-convex} functions.
\begin{definition}
 A function $\phi:\R^n\to\R$ is said to be Schur-convex if $x\prec y$ implies $\phi(x)\leq \phi(y)$.
\end{definition}$ $\\
Schur-convexity works with weak majorization as follows
\begin{proposition}[see 3.A.8 in \cite{marshall2011inequalities}]\label{Schur_con2}
 A real valued function $\phi$, defined on $\R^n$ satisfies
 \[
  x\prec_w y \Rightarrow \phi(x)\leq \phi(y)
 \]
 if and only if $\phi$ is Schur-convex and componentwise increasing. Analogously,
 \[
  x\prec^w y \Rightarrow \phi(x)\leq \phi(y)
 \]
 holds if and only if $\phi$ is Schur-convex and componentwise decreasing.
\end{proposition}$ $\\
An important class of Schur-convex functions is given by the following proposition.
\begin{proposition}[see 3.C.1 in \cite{marshall2011inequalities}]\label{Schur_con1}
 If $\psi:\R\to\R$ is convex, then the function $\phi$ defined by
 \[
  \phi(x)=\sum_{i=1}^n\psi(x_i)
 \]
is Schur-convex. Obviously, if $\psi$ is in addition increasing (decreasing) then $\phi$ is Schur-convex and componentwise increasing (decreasing).
\end{proposition}$ $\\
Another natural question would be, which operations preserve majorization. It turns out that the ordering of vectors plays an important role here. A famous result for oppositely ordered vectors is the following.
\begin{proposition}[see 6.A.2 in \cite{marshall2011inequalities}]\label{opposite_ordering_sum}
 For any two vectors $x^{(1)}$ and $x^{(2)}$ on $\R^n$, it holds that
 \[
  x^{(1)}_\downarrow+x^{(2)}_\uparrow\prec x^{(1)}+x^{(2)}.
 \]
\end{proposition}$ $\\
It is also possible to consider more general aggregation operators than the sum. This however requires more assumptions and yields only weak majorization (which is completely sufficient for our purpose).
\begin{proposition}[see 6.C.4 in \cite{marshall2011inequalities}]\label{opposite_ordering_gen}
 \[
  \left(h(x^{(1)}_{[1]},x^{(2)}_{[n]}),\dots,h(x^{(1)}_{[n]},x^{(2)}_{[1]})\right)\prec_w\left(h(x^{(1)}_1,x^{(2)}_1),\dots,h(x^{(1)}_n,x^{(2)}_n)\right)
 \]
 holds for any two vectors $x^{(1)}$ and $x^{(2)}$ in $\R^n$, if and only if $h$ is supermodular and either increasing in each component or decreasing in each component.
\end{proposition}$ $\\
These statements are clearly restricted to two vectors since ``Oppositely ordered'' does not make sense otherwise and this is also the reason why the decomposition property \eqref{decompose_property} is needed.\\
We are now ready to prove Theorem \ref{weak_statement}.\\
\begin{proof}[Proof of Theorem \ref{weak_statement}:]
 We want to show that for every matrix $\tilde{X}\in\Per(X)\setminus\Opp_h(X)$, there exists a matrix $\hat{X}\in\Opp_h(X)$ such that
 \begin{equation}\label{want_to_prove}
  \sum_{k=1}^nf(\hat{x}^{(1)}_k,\dots,\hat{x}^{(d)}_k)\leq\sum_{k=1}^nf(\tilde{x}^{(1)}_k,\dots,\tilde{x}^{(d)}_k)
 \end{equation} %with $f$ as in Theorem \ref{weak_statement} ??
 If $\tilde{X}\in\Per(X)\setminus\Opp_h(X)$, then there exists a column $\tilde{x}^{(i)}$ which is not oppositely ordered to $h^{d-1}(\tilde{X}_{-i})$. Denote by $\hat{X}$ the matrix obtained from $\tilde{X}$ by ordering $\tilde{x}^{(i)}$ oppositely to $h^{d-1}(\tilde{X}_{-i})$. By Proposition \ref{opposite_ordering_gen} it holds that
 \begin{align*}
  h(\hat{X})&=\left(h^2\big(\hat{x}^{(i)}_1,h^{d-1}(\hat{X}_{-i})_1\big),\dots,h^2\big(\hat{x}^{(i)}_n,h^{d-1}(\hat{X}_{-i})_n\big)\right)^T\\
  &\prec_w\left(h^2\big(\tilde{x}^{(i)}_1,h^{d-1}(\tilde{X}_{-i})_1\big),\dots,h^2\big(\tilde{x}^{(i)}_n,h^{d-1}(\tilde{X}_{-i})_n\big)\right)^T=h(\tilde{X}). %\hat{x} and \tilde{X} or [i] notation? Maybe vectorized?
 \end{align*}
Since $g$ is increasing and convex, we get with the help of Proposition \ref{Schur_con1} that $\sum_{k=1}^nf(x^{(1)}_k,\dots,x^{(d)}_k)=\sum_{k=1}^ng(h(X)_k)$ is Schur-convex and increasing, hence, \eqref{want_to_prove} holds by Proposition \ref{Schur_con2}. If $\hat{X}\in\Opp_h(X)$, we are done. It remains to show that this procedure eventually reaches an element in $\Opp_h(X)$. To see this, we look at the reordering of a certain column in more detail. Fix $i$ and let $K\subset\{1,\dots,n\}$ be the set of all indices that appear in a pair which violates the opposite ordering of $\tilde{x}^{(i)}$ and $h^{d-1}(\tilde{X}_{-i})$. So $k_1,k_2\in K\iff \tilde{x}^{(i)}_{k_1}< (>) \tilde{x}^{(i)}_{k_2}$ and $h^{d-1}(\tilde{X}_{-i})_{k_1} < (>) h^{d-1}(\tilde{X}_{-i})_{k_2}$. We construct $\hat{x}^{(i)}$ in the following way: let $x^{(i)}_{\underline{k}}:=\min_{k\in K} x_k$ and let $h^{d-1}(\tilde{X}_{-i})_{\overline{k}}:=\max_{k\in K} h^{d-1}(\tilde{X}_{-i})_{k}$. Since $\underline{k},\overline{k}\in K$ and by construction, we know that
 \begin{align}\label{monobeweis}
  x^{(i)}_{\underline{k}}< x^{(i)}_{\overline{k}}\quad\text{ and }\quad h^{d-1}(\tilde{X}_{-i})_{\underline{k}}< h^{d-1}(\tilde{X}_{-i})_{\overline{k}}.
 \end{align}
Now exchange $x^{(i)}_{\underline{k}}$ and $x^{(i)}_{\overline{k}}$. It is not hard to see that this reduces the number of violating indices by at least one. We repeat this procedure until $K=\emptyset$ and have thus created a vector $\hat{x}^{(i)}$ which is ordered oppositely to $h^{d-1}(\tilde{X}_{-i})$. At this point it is important to see that due to \eqref{monobeweis} and the monotonicity of $h^{2}$, we have
\[
  h^{2}(\tilde{x}^{(i)}_{\overline{k}},h^{d-1}(\tilde{X}_{-i})_{\overline{k}})>h^{2}(\tilde{x}^{(i)}_{\underline{k}},h^{d-1}(\tilde{X}_{-i})_{\overline{k}})>h^{2}(\tilde{x}^{(i)}_{\underline{k}},h^{d-1}(\tilde{X}_{-i})_{\underline{k}})
\]
and
\[
 h^{2}(\tilde{x}^{(i)}_{\overline{k}},h^{d-1}(\tilde{X}_{-i})_{\overline{k}})>h^{2}(\tilde{x}^{(i)}_{\overline{k}},h^{d-1}(\tilde{X}_{-i})_{\underline{k}})>h^{2}(\tilde{x}^{(i)}_{\underline{k}},h^{d-1}(\tilde{X}_{-i})_{\underline{k}}).
\]
So after each exchanging step of the above type, the values of $h(\tilde{X})$ have changed in exactly two components and the new values are strictly between the old ones. Hence it is clear that $h(\hat{X})\neq h(\tilde{X})$ and also that $h(\hat{X})$ is not a permutation of $h(\tilde{X})$. However, $x\prec_w y$ and $y\prec_w x$ at the same time implies that $x$ is a permutation of $y$, so we know that $h(\tilde{X})\prec_w h(\hat{X})$ cannot hold. This means that $h(\hat{X})$ is strictly below $h(\tilde{X})$ w.r.t $\prec_w$. Since $\Per(X)$ is finite, it follows that after a finite number of steps, we arrive at a matrix $\hat{X}\in\Opp_h(X)$ which satisfies 
\[
  \sum_{k=1}^nf(\hat{x}^{(1)}_k,\dots,\hat{x}^{(d)}_k)\leq\sum_{k=1}^nf(\tilde{x}^{(1)}_k,\dots,\tilde{x}^{(d)}_k).
\]
\end{proof}
\begin{remark}
With the proof of Theorem \ref{weak_statement} we actually showed that for any element $\tilde{X}\in\Per(X)$, there is a chain of matrices $\hat{X}_{(1)},\hat{X}_{(2)},\dots,\hat{X}_{(n)}$ such that $\hat{X}_{(n)}\prec_w\hat{X}_{(n-1)}\prec_w\dots\prec_w\hat{X}_{(1)}\prec_w\tilde{X}$ and $\hat{X}_{(n)}\in\Opp_h(X)$. Indeed, weak majorization is much stronger than just the inequality \eqref{want_to_prove}. This is also reflected by the fact that \eqref{want_to_prove} can be installed without majorization, using a weaker statement by Lorentz, see \cite{puccetti2015expectation} for details. There it is also shown that using Lorentz to obtain \eqref{want_to_prove} does not require $h$ to be monotone at all, whereas we showed, weak majorization needs (non strict) monotonicity. However, since strict monotonicity of $h$ is, in general, required to reach an element in $\Opp_h(X)$ the weaker statement does not give any benefit.
\end{remark}
\begin{remark}
The proof also shows, why Theorem \ref{Pucc2.4} does not require $g$ to be monotone: According to Theorem \ref{opposite_ordering_sum}, the sum as aggregation function yields strong majorization, which is why the inequality \eqref{want_to_prove} holds for $f=g\circ h$ for arbitrary convex $g$. A complete description of all aggregating functions yielding strong monotonicity and hence \eqref{want_to_prove} without monotonicity of $g$ is given in the following statement.
\end{remark}
\begin{proposition}[see 6.B.2 in \cite{marshall2011inequalities}]\label{strong_major_cond}
 \[
  \left(h(x^{(1)}_{[1]},x^{(2)}_{[n]}),\dots,h(x^{(1)}_{[n]},x^{(2)}_{[1]})\right)\prec\left(h(x^{(1)}_1,x^{(2)}_1),\dots,h(x^{(1)}_n,x^{(2)}_n)\right)
 \]
 holds for any two vectors $x^{(1)}$ and $x^{(2)}$ in $\R^n$, if and only if $h$ is of the form $h(x_1,x_2)=\varphi_1(x_1)+\varphi_2(x_2)$, where $\varphi_1$ and $\varphi_2$ are monotone in the same direction.
\end{proposition}$ $\\
At this point, we would like to stress the fact that the gain of Theorem \ref{weak_statement} over Theorem \ref{Pucc2.6} consists of dropping the strict monotonicity of the overall function $f$ by showing that strict monotonicity is only needed for the aggregation function $h$. An example of a function that is included in Theorem \ref{weak_statement} but not in \ref{Pucc2.6} would be the stop-loss functional $f(x_1,\dots,x_d)=\max(x_1+\dots+x_d-k,0)$ for some $k\neq 0$. Except for the issue of strict monotonicity, the composition with an increasing convex function was already possible in \ref{Pucc2.6} as the next lemma shows.
\begin{lemma}
 If $h:\R^d\to\R$ has the decomposition property \eqref{decompose_property} then for any $g:\R\to\R$ the composition $f=g\circ h$ also satisfies \eqref{decompose_property}. If furthermore the decomposition of $h$ is supermodular (in particular $h^2$ has to be supermodular) and $g$ is increasing and convex, then the decomposition of $f$ is also supermodular (in particular $f^2$ is supermodular). %explicitly mention that Pucc & Rü were wrong here?
\end{lemma}
\begin{proof}
 It holds
 \begin{align*}
 f(x_1,\dots,x_d)&=(g\circ h)(x_1,\dots,x_d)=g\left(h((x_1,\dots,x_d)\right)\\
 &=g\left(h^2(x_i,h^{d-1}(x_1,\dots,x_{i-1},x_{i+1},\dots,x_d))\right)\\
 &=(g\circ h^2)(x_i,h^{d-1}((x_1,\dots,x_{i-1},x_{i+1},\dots,x_d)))
 \end{align*}
 for $i=1,\dots,d$. So take $f^2=g\circ h^2$ and $f^{d-1}=h^{d-1}$ and we have \eqref{decompose_property}. For the claim that $f^2$ is supermodular if $h^2$ is supermodular and $g$ is increasing and convex, use Proposition 6.D.2 from \cite{marshall2011inequalities}. 
\end{proof}
$ $\\
Using majorization, we were able to identify new cases where the RA can be applied to compute bounds on expected values. In particular, this extension yields more flexibility when working with non-linear risk aggregation. While these models are relatively sparse in mathematical finance, they enjoy some popularity in the field of modeling medical risks (for example see \cite{brattin1994quantitative} or \cite{guo2000familial}). Another, even broader, extension is possible, when we generalize the decomposition property as follows.\\
$ $\\
Suppose that for every index $i=1,\dots,d$, $h:\R^d\to\R$ can be decomposed into two functions $\prescript{}{i}h^2:\R^2\to\R$ and $\prescript{}{-i}h^{d-1}:\R^{d-1}\to\R$ such that for all $x\in\R^d$
\begin{equation}\label{decompose_property_strong}
 h(x_1,\dots,x_d)=\prescript{}{i}h^2(x_i,\prescript{}{-i}h^{d-1}(x_1,\dots,x_{i-1},x_{i+1},\dots,x_d)).
\end{equation}
Note that in contrast to \eqref{decompose_property}, here the decomposition may depend on the index. We want to use the notations $h(X)$ and $\prescript{}{-i}h^{d-1}(X_{-i})$ just as before. So for all $i=1,\dots,d$
\[
 h(X)=\begin{pmatrix}
       h(x^{(1)}_1, & x^{(2)}_1, & \dots, & x^{(d)}_1)\\
       h(x^{(1)}_2, & x^{(2)}_2, & \dots, & x^{(d)}_2)\\
       \vdots & \vdots &  & \vdots\\
       h(x^{(1)}_n, & x^{(2)}_n, & \dots, & x^{(d)}_n)
      \end{pmatrix}
      =\begin{pmatrix}
        \prescript{}{i}h^2\big(x^{(i)}_1, & \prescript{}{-i}h^{d-1}(X_{-i})_1\big)\\
        \prescript{}{i}h^2\big(x^{(i)}_2, & \prescript{}{-i}h^{d-1}(X_{-i})_2\big)\\
        \vdots\\
        \prescript{}{i}h^2\big(x^{(i)}_n, & \prescript{}{-i}h^{d-1}(X_{-i})_n\big)
       \end{pmatrix}.
\]
Notice that after a column $x^{(i)}$ is fixed, the same decomposition with index $i$ is applied in every row.\\
$ $\\
We now restate Theorem \ref{weak_statement} in generalized form.
\begin{theorem}\label{strong_statement}
 If $f(x^{(1)},\dots,x^{(d)})=g(h(x^{(1)},\dots,x^{(d)}))$ where $g:\R\to\R$ is increasing and convex and $h:\R^d\to\R$ is supermodular, componentwise strictly monotonic and satisfies \eqref{decompose_property_strong} with all $\prescript{}{i}h^2$ also being supermodular, then 
 \[
  m_f(X)=\min_{\tilde{X}\in\Opp_h(X)}\sum_{k=1}^nf(\tilde{x}^{(1)}_k,\dots,\tilde{x}^{(d)}_k).
 \]
 Furthermore, the requirement of $g$ being increasing can be dropped, whenever $\prescript{}{i}h^2$ is of the form $\prescript{}{i}h^2(x_1,x_2)= \varphi_1(x_1)+\varphi_2(x_2)$, where $\varphi_1$ and $\varphi_2$ are monotone in the same direction.
\end{theorem}
\begin{proof}
Since the majorization does not depend on the decomposition to be the same for each index, the proof of Theorem \ref{weak_statement} carries over verbatim. The second statement follows from \ref{strong_major_cond}.
\end{proof}
$ $\\
Theorem \ref{strong_statement} contains the statements \ref{weak_statement} and \ref{Pucc2.4} as special cases. Note that in contrast to \eqref{decompose_property}, a function does not have to be symmetric to satisfy \eqref{decompose_property_strong}. An important case that is included in Theorem \ref{strong_statement} but not in \ref{weak_statement} or \ref{Pucc2.4} is a weighted sum.
\section{Applications}
To illustrate the usefulness of the new cases the RA can be applied to, we want to compute some examples. That is, given some marginal distributions $F_1,\dots,F_d$ and a cost function $f:\R^d\to\R$ which fulfills our assumptions, we want to estimate
\begin{equation}\label{infexpdef}
 s_f:=\inf\left\{\E\left[f(X^{(1)},\dots,X^{(d)})\right]:X^{(i)}\sim F^{(i)}, 1\leq i\leq d\right\}.
\end{equation}
Here, we are only interested in the infimum since it is well known that for supermodular cost functions $f$, the supremum
\[
S_f:=\sup\left\{\E\left[f(X^{(1)},\dots,X^{(d)})\right]:X^{(i)}\sim F^{(i)}, 1\leq i\leq d\right\}
\]
is attained, when $X_1,\dots,X_d$ are comonotonic (i.e. their copula is the upper Fr\'{e}chet-Hoeffding bound). Hence $S_c$ is known in this case.\\
$ $\\
Considerig what we have seen in the preceding section, there are two immediate obstacles. The first one is that so far, we always talked about $n$-discrete distributions and now we want to work with general distributions $F_i$. This can be adressed by working with the two $n$-discrete distributions
\[
 \underline{F}_i=\frac1n\sum_{k=0}^{n-1}1_{[q^i_k,\infty)}(x)\quad\text{and}\quad \overline{F}_i=\frac1n\sum_{k=1}^{n}1_{[q^i_k,\infty)}(x)
\]
where the $q^i_k$ are defined in terms of the quantile functions $F_i^{-1}$ by $q^i_k:=F_i^{-1}(\frac{k}{n})$. Writing $\underline{s}_f$ resp. $\overline{s}_f$ for \eqref{infexpdef} with the $F_i$ replaced by $\underline{F}_i$ resp. $\overline{F}_i$, it holds that if the cost function $f$ is componentwise increasing we have
\[
\underline{s}_f\leq s_f\leq\overline{s}_f.
\]
This allows us to compute a range for $s_f$ that will become small when $n$ increases.\\  
$ $\\
The second problem is that the rearrangement methodology aims to find $m_f$ and not $s_f$. The difference here is that for $m_f$ only distributions that give equal mass of $\frac1n$ to $n$ out of $n^d$ possible realizations of $X$ are considered. However, % These distributions are related to \textit{Shuffles of Min}, a special class of copulas and
for $n$ large enough, we have
\[
 s_f\approx \frac{m_f}{n}.
\]
Combining these two ideas leads to the approximation of $s_f$ by the RA. For more details see section 3 in \cite{puccetti2015expectation}. Note that the requirement to have $f$ componentwise increasing for this method eliminates the increased generality, \ref{Pucc2.4} had over \ref{Pucc2.6} and \ref{weak_statement} in terms of monotonicity.\\
\subsection{A weighted portfolio}
Imagine a portfolio consisting of three different assets with value processes $X_1(t)$, $X_2(t)$ and $X_3(t)$. The value of the portfolio at time $t=0$ is given by
\[
 L(0)=\alpha_1X_1(0)+\alpha_2X_2(0)+\alpha_3X_3(0)
\]
where $\alpha_i$ denotes the amount of positions held in asset $X_i$. Furthermore, the return of asset $i$ is given by $R_i:=\frac{X_i(1)-X_i(0)}{X_i(0)}\sim F_i$ with known distribution functions $F_i$. Then the return of the portfolio at time $t=1$ is given by
\[
 h(R_1,R_2,R_3)=\frac{L(1)-L(0)}{L(0)}=\frac{\alpha_1 X_1(0)}{L(0)} R_1 + \frac{\alpha_2 X_2(0)}{L(0)} R_2 + \frac{\alpha_3 X_3(0)}{L(0)} R_3 .
\]
So interpreting $w_i:=\frac{\alpha_i X_i(0)}{L(0)}\in[0,1]$ as the fraction of the original wealth that was invested in $X_i$, we get $h(R_1,R_2,R_3)=w_1 R_1+w_2 R_2 + w_3 R_3$.\\
Assuming a guaranteed return of $k$ on the portfolio leaves us to examine\\
$\E\left[g(h(R_1,R_2,R_3))\right]$ with $g(x)=\max(x-k,0)$. Obviously $g$ and $h$ satisfy the assumptions of Theorem \ref{strong_statement}, hence we can use the RA to compute $s_f$ for $f=g\circ h$.\\
$ $\\
We considered varying sets of distributions $F_1$, $F_2$, $F_3$. Since the cost function $f$ is unbounded from above, the discretizations $\overline{F}_i$ cannot be used in a meaningful way with distributions that are unbounded from above. Therefore, we truncated all distributions with infinite support at the $0.001\%$ quantile and marked the modified distributions with an asterisk $(\ast)$. Note that in practice, the question where to truncate can be hard to decide, especially when dealing with heavy tails, see e.g. \cite{clark2013note} for considerations concerning the Pareto distribution.\\
For consistency, also the lower bounds of (positively) unbounded distributions were calculated with the truncated distribution. To have a reference value, we also computed the range of $\E\left[g(h(R_1,R_2,R_3))\right]$ by solving a linear program. This procedure yields rigorous bounds but is computationally more costful and thus usually cannot provide the accuracy of the RA. For more details on the LP method and a comparison of the two approaches see \cite{michael2016bounds} and the references therein.\\
$ $\\
The RA was applied with a grid of $10^5$ sections in each dimension, whereas the LP method used a grid of 60 sections in each dimension. The weight vector was set to $w:=(0.5,0.2,0.3)$ throughout these computations.
\begin{table}[!h]
\centering
 \begin{tabular}{llllll}
  $F_1$ & $F_2$ & $F_3$ & range LP & range RA\\
  \hline
  $\mathcal{U}([0,0.4])$ & $\mathcal{U}([0.1,0.5])$ & $\mathcal{U}([0,1])$ & $0.0058$-$0.0148$ & $0.0099$-$0.0100$\\
  $exp(1)^\ast$ & $exp(2)^\ast$ & $exp(4)^\ast$ & $0.3416$-$0.4711$ & $0.3749$-$0.3750$\\
  $\mathcal{U}([0,0.4])$ & $exp(3)^\ast$ & $\mathcal{U}([0,1])$ & $0.0092$-$0.0303$ & $0.0166$-$0.0167$\\
  $exp(1)^\ast$ & $Pareto(2)^\ast$ & $\mathcal{N}(0,0.25)^\ast$ & $0.2876$-$1.4912$ & $0.3990$-$0.4054$\\
  %$\mathcal{U}([0,1])$ & $\mathcal{U}([0,1])$ & $Beta(2,5)$ & $0.030489$ & $0.033876$ & $0.037232$ \\
  %$exp(1)$ & $\mathcal{U}([0,1])$ & $Beta(2,5)$ & $0.044647$ & $0.051213$ & $0.058365$ \\
  \hline
 \end{tabular}
\caption{Approximation results for inhomogeneous marginals. LP values as a reference.\label{tab_inhom_marg}}
\end{table}\\
$ $\\
For the LP approach, the computation time depends strongly on the chosen marginal distributions and was on average around $10$ minutes per value. The RA generally took less than one minute to compute both upper and lower bound. All results were obtained using the open source language \textbf{R}, where we used Marius Hofert's implementation of the RA from the package \textit{qrmtools}\cite{qrmtools}.
%In the preceding section, we always talked about $n$-discrete distributions

\bibliographystyle{chicago}

\end{document}